\documentclass[11pt]{article}
\usepackage{amsmath, amsthm}
\usepackage{amsmath}
\usepackage{amsfonts}
\usepackage{amssymb}
\usepackage{mathrsfs} 
\usepackage{verbatim}
\usepackage[pdftex]{graphicx}
\usepackage{epsfig}
\usepackage{amscd}
\usepackage{amsthm}
\usepackage{color}

\numberwithin{equation}{section}

\newtheorem*{mainthm}{Main Theorem}

\newtheorem{thm}{Theorem}[section]
\newtheorem{cor}[thm]{Corollary}
\newtheorem{lemma}[thm]{Lemma}
\newtheorem{prop}[thm]{Proposition}

\newtheorem{remark}[thm]{Remark}
\newtheorem*{remmark}{Remark}

\theoremstyle{definition}

\newtheorem{defn}[thm]{Definition}

\newcommand{\R}{\mathbb{R}}
\newcommand{\e}{{\varepsilon}}

\newcommand{\hess}{\text{Hess\hspace{2pt}}}

\begin{document}
\title{Degenerate diffusion with a drift potential: a viscosity solutions approach}
\author{I.~C.~Kim$^1$, H.~K.~Lei$^1$}
\vspace{-2cm}
\date{}
\maketitle
\vspace{-4mm}
\centerline{${}^1$\textit{Department of Mathematics, UCLA}}
 
\begin{quote}
{\footnotesize {\bf Abstract: } We introduce a notion of viscosity solutions for a  nonlinear degenerate diffusion equation with a drift potential. We show that our notion of solutions coincide with the weak solutions defined via integration by parts. As an application of the viscosity solutions theory, we show that the free boundary uniformly converges to the equilibrium as $t$ grows. In the case of a convex potential, an exponential rate of free boundary convergence is obtained. }
\end{quote}

\section{\large{Introduction}}

Consider a $C^2$ function $\Phi(x):\R^n\to \R$, and consider a nonnegative, continuous function $\rho_0(x):\R^n\to \R$ which has compact support $\Omega_0$. In this paper  we study the porous medium equation with a drift
\begin{equation} \label{main_eq_density}\rho_t = \triangle(\rho^m) + \nabla \cdot (\rho \nabla \Phi),\end{equation}
for $m > 1$, with initial data $\rho_0(x)$. Note that, at least formally (and proven in \cite{BH} for the weak solutions), the solution of (\ref{main_eq_density}) preserves its $L^1$ norm.

 It will be convenient to change from the density variable to the pressure variable 
\begin{equation}\label{density}
u = \frac{m}{m-1}\rho^{m-1}, \quad u_0 =\frac{m}{m-1}\rho_0^{m-1}.
\end{equation}
so that the equation becomes 
$$
u_t = (m-1) u \triangle u + |\nabla u|^2 + \nabla u \cdot \nabla \Phi + (m-1) u \triangle \Phi\leqno\mbox{(PME--D)}
$$
(for more on the density to pressure transform see e.g., the discussions in \cite{BV}).  We consider continuous and nonnegative solutions in the space--time domain $Q = \mathbb R^n \times (0, T)$ for some $T > 0$, with prescribed initial conditions $u(x, 0) = u_0(x) \in C(\mathbf R^n)$.  

When $\Phi\equiv  0$, (PME--D) is the widely-studied Porous Medium Equation (PME): We refer to the book \cite{V} for the references. 
Moreover when $V=|x|^2$, (PME--D) is obtained as a re-scaled version of the (PME) with new variables  
$$
\theta(\eta,\tau):=t^{-\alpha}u(x,t), \eta = xt^{-\beta} , \tau = \ln t.
$$
where $u$ solves (PME). This suggests that the local behavior of (PME--D) is similar to that of (PME), with perturbations due to the inhomogeneity of $\Phi$. We will illustrate this fact in the construction of various barriers in section 2 and 3.

The weak solution theory for (PME--D) in the case of bounded domains has been developed in \cite{BH} and \cite{DiB}.  Also, in \cite{JGJ}, existence and uniqueness of solutions are established for the full space case under reasonable assumptions (either the initial data is compactly supported or the potential has less than quadratic growth at infinity). 

Further, uniform convergence to equilibrium for (PME--D) have also been shown in \cite{BH} (see Theorem 3.1).  In \cite{JT}, the connection between the (PME) and the nonlinear Fokker--Planck equation is established, which facilitates the use of the entropy method to derive an explicit $L^1$ rate of convergence.  In \cite{jose_et_al}, an extensive study is made of a general form of the nonlinear Fokker--Planck equation, i.e., $\rho_t = \nabla \cdot (\nabla\varphi(\rho) + \rho \nabla V)$ with suitable assumptions on $\varphi$ and exponential $L^1$ rate of convergence is obtained.  (PME--D) falls under the framework of \cite{jose_et_al}, and in fact it is the case that almost all of our results would also go through for a general equation of this form, but for ease of exposition we will restrict attention to (PME--D).

We introduce a notion of viscosity solution  for the free boundary problem associated with this equation, which we will show to be equivalent to the usual notion of weak solutions -- see \cite{CIL} for the general theory of viscosity solutions. Note that, formally, the  {\it free boundary}  $\partial\{u>0\}$ moves with the outward normal velocity
$$
V= \frac{u_t}{|Du|} = (Du+D\Phi)\cdot \frac{Du}{|Du|} = |Du| +D\Phi\cdot\frac{Du}{|Du|},
$$
where the first equality is due to the fact that $u=0$ on $\Gamma(u)$.  In this regard we closely follow the framework and arguments set out in \cite{CV} (see also \cite{Kim} and \cite{BV}), where the viscosity concept is introduced and studied for the Porous Medium Equation.  We point out especially that \cite{BV} extends the result of \cite{CV} to the case where the diffusion term is multiplied by more general nonlinearities; our focus, however, is on the added drift term, which introduces spatial inhomogeneities.  The key utility of the viscosity concept here is that we will be able to describe the pointwise behavior of the free boundary evolution by maximum principle arguments with local barriers. As an application, we are able to extend the results of \cite{BH} and \cite{jose_et_al} to a stronger notion of free boundary convergence.

\begin{mainthm}\label{main_thm}
There exists a viscosity solution of $u$  (PME--D) with $\int (\frac{m-1}{m}u_0)^{1/m-1} dx=\int \rho_0(x) dx = m_0$. Further:
\begin{itemize}
\item[(a)] $u$ is unique and coincide with weak solutions studied in \cite{BH} and \cite{jose_et_al}.\\ 
\item[(b)] If $|D\Phi|>0$ except at $x=x_0$ where $\Phi$ achieves its minimum, then from \cite{BH} there exists a unique $C_0>0$ depending only on $m_0$ such that $u$ uniformly converges to $u_\infty:=(C_0-\Phi(x))_+$. Then we have
$$
\sup_{y\in\partial\{u(\cdot,t)>0\}} d(y,\Gamma(u_\infty)) \to 0 \hbox{ as } t\to\infty.
$$
\item[(c)] For a strictly convex $\Phi$ with $k_0 <\Delta \Phi $, there exists $K,\alpha>0$ depending on $u_0$, $k_0$, the $C^2$-norm of $\Phi$  and $n$ such that
$$
\sup_{y\in\partial\{u(\cdot,t)>0\}}d(y,\Gamma(u_\infty)) \leq Ke^{-\alpha t}.
$$

\end{itemize}
\end{mainthm}

\vspace{10pt}

\begin{remmark}
1. We point out that the free boundary convergence may not hold if $|D\Phi|$ vanishes at some points, even though the uniform convergence of the solution still holds.

2. In the case of $\Phi(x)=|x|^2$ (that is for the renormalized (PME))  Lee and Vazquez \cite{LV} showed that the interface becomes convex in finite time.  It is unknown whether such results hold for general convex potentials: we shall investigate this in an upcoming work.

\end{remmark}

\section{Viscosity Solution}

In this section we introduce the appropriate notion of viscosity solution for (PME--D) and show that it is equivalent to the usual notion of weak solution.  Our definition descends from those in \cite{CV} and \cite{Kim}.  For more details we also refer the reader to the definitions, discussions and results in \cite{BV}.  

\subsection{Definition and Basic Properties}

Let $Q:= \R^n\times (0,\infty)$.  For a nonnegative function $u(x,t)$ in $Q$, we denote the {\it positive phase}
$$
\Omega(u) =\{u>0\},\quad \Omega_t(u):=\{x: u(x,t)>0\}
$$ 
and the {\it free boundary } 
$$
\Gamma(u)=\partial\Omega(u), \quad \Gamma_t(u):=\partial\Omega_t(u).
$$
As in \cite{CV}, to describe the free boundary behavior using comparison arguments we need an appropriate class of test functions to deal with the degeneracy of (PME--D).

Let $\Sigma$ be a smooth, cylinder--like domain in $\R^n\times [0,\infty)$, i.e., 
\begin{equation}\label{domain}
\Sigma = \bigcup_{t_1\leq t\leq t_2} \Sigma(t)\times\{t\}, \hbox{ where } \Sigma(t)\hbox{ is a smooth domain in } \R^n.
\end{equation}

\begin{defn}\label{free_bound_soln}
A nonnegative function $u \in C^{2,1}(\overline{\{u>0\}}\cap\Sigma)$ is a {\it classical free boundary subsolution} in $\Sigma$ if $u$ satisfies (PME--D) with $\leq$ replacing $=$ in the classical sense in $\{ u > 0\}\cap\Sigma$,  $|Du|>0$  on $\Gamma(u)\cap\Sigma$ with the outward normal velocity
$$
V \leq \beta|\nabla u| + \nabla \Phi \cdot \frac{\nabla u}{|\nabla u|}\hbox{ on } \Gamma(u)
$$
or equivalently, 
$$
 u_t \leq \beta|\nabla u|^2 + \nabla \Phi \cdot \nabla u\hbox{ on } \Gamma(u).
 $$
 
 \medskip
 
We define a {\it classical free boundary supersolution} by replacing $\leq$ with $\geq$.  Finally, $u$ is a \emph{classical free boundary solution} if it is both a sub and supersolution.  
\end{defn}

Before proceeding further it is convenient to introduce some auxiliary definitions.  

\begin{defn}
Let $\varphi$ be a continuous, nonnegative function.  Now if $\psi$ is another such function,  then we say that $\varphi$ {\it touches} $\psi$ {\it from above} at $(x_0, t_0)$ in $\Sigma$ if $\varphi - \psi$ has a local minimum zero at $(x_0,t_0)$ in $\Sigma \cap \{ t \leq t_0\}$.  We have a similar definition for $\varphi$ {\it touching } $\psi$ {\it from below}.
\end{defn}

\begin{defn}[Strictly Separated Data]
  For two nonnegative functions $u, v: \R^n\to \R$,  we write $u_0 \prec v_0$ if the following holds:
  $\mbox{supp}(u_0)$ is compact and $\mbox{supp}(u_0) \subset \mbox{Int(supp(}v_0))$ and inside $\mbox{supp}(u_0)$, $u_0(x) < v_0(x)$.   
\end{defn}

  We note that e.g., due to the maximum principle, a classical free boundary subsolution that lies below a classical free--boundary supersolution at time $t_1 \geq 0$ cannot cross the supersolution from below at a later time $t_2 > t_1$.  This observation leads to a notion of viscosity solution which takes into account the free boundary: 
\begin{defn}\label{viscosity}
Let $u$ be a continuous, nonnegative function in $Q$.
\begin{itemize}
\item[$\circ$] $u$ is a \emph{viscosity subsolution} of (PME--D) if, for any given smooth domain $\Sigma$ given in (2.1), for every $\varphi \in C^{2, 1}(\Sigma)$ that touches $u$ from above at the point $(x_0, t_0)$, we have 
\begin{equation}\label{visc_defn}\varphi_t \leq \alpha \varphi \triangle \varphi + \beta |\nabla \varphi|^2 + \nabla \cdot (\varphi \nabla \Phi)\end{equation}

\item[$\circ$] $u$ is a \emph{viscosity supersolution} of (PME--D) if, for any given smooth domain $\Sigma$ given in (2.1),\\
(i) for every $\varphi \in C^{2, 1}(\Sigma)$ that touches $u$ from below at the point $(x_0, t_0) \in \Omega(u)\cap \Sigma$, we have 
\begin{equation}\label{visc_defn}\varphi_t \geq \alpha \varphi \triangle \varphi + \beta |\nabla \varphi|^2 + \nabla \cdot (\varphi \nabla \Phi)\end{equation}

(ii) for every classical free--boundary subsolution $\varphi$ in $\Sigma$, the following is true:
If $ \varphi \prec u$ on the parabolic boundary of $\Sigma$, then $\varphi \leq u$ in $\Sigma$.  That is, every classical free--boundary subsolution that lies below $u$ at a time $t_1 \geq 0$ cannot cross $u$ at a later time $t_2 > t_1$.\\

\item[$\circ$] $u$ is a {\it viscosity solution} of (PME--D) with initial data $u_0$ if $u$ is both a super-- and subsolution and $u$ uniformly converges to $u_0$ as $t\to 0$.

\end{itemize}
\end{defn}

\begin{remark} In general one can define viscosity sub-- and supersolutions respectively  as upper-- and lower semicontinuous functions. Such a definition turns out to be useful when one cannot verify continuity of solutions obtained via various limits. This problem does not arise in our investigation here thanks to \cite{BH}, and our definition assumes continuity of solutions.  
\end{remark}

It is fairly straightforward to verify that a classical free boundary sub-- (super)solution is also a viscosity sub-- (super)solution.

\begin{lemma}\label{classical_visc}
If $w$ is a classical free boundary sub-- (super) solution to (PME--D), then $w$ is also a viscosity sub-- (super) solution.
\end{lemma}
\begin{proof}
We will be brief: The subsolution case presents no difficulty since if contact with some $\varphi \in C^{2, 1}(\Sigma)$ occurs in $\Omega(w)$ then we use the fact that $w$ is classical there, whereas no contact can occur on the free boundary unless $|\nabla w| = 0$, in which case the differential inequality is satisfied since then $\varphi = |\nabla \varphi| = 0$ and $\varphi_t \leq 0$.  

If $w$ is a classical free boundary supersolution, then (i) in Definition \ref{viscosity} follows as before.  To see (ii), let us note that if $\varphi$ is a classical free boundary subsolution which crosses $w$, then since the free boundary is $C^2$, Hopf's Lemma implies that at the touching point $|\nabla \varphi| < |\nabla w|$ (see e.g., \cite{max_principle}).  On the other hand, since $\varphi$ started below $w$, at the touching point we must have $v_n(\varphi) \geq v_n(w)$, which leads to a contradiction since it is also the case that we have $\frac{\nabla w}{|\nabla w|} = \frac{\nabla \varphi}{|\nabla \varphi|}$.

\end{proof}

Next we have the following stability result.

\begin{lemma}\label{stability}
Let $u^\e$ be a smooth solution of (PME--D) with initial data $u_0+\e$ and let $u$ be its uniform limit.
Then $u$ is a viscosity solution of (PME--D) with initial data $u_0$. 
\end{lemma}
\begin{proof}
Let $\Sigma$ be as given in (\ref{domain}) and let  $\varphi\in C^{2,1}(\Sigma)$. 

\medskip

1. Let us first show that $u$ is a subsolution.  First suppose that $\varphi$ touches $u$ from above at the point $(x_0, t_0)$.  We may assume that $u - \varphi$ has a strict maximum at $(x_0, t_0)$ in $\overline \Omega(u) \cap \Sigma \cap \{t \leq t_0\}$ by replacing $\varphi$ by 
\[ \tilde \varphi(x, t): = \varphi(x, t) + \sigma( (x-x_0)^4 - (t- t_0)^2), ~~~\sigma > 0\]
if necessary.  By uniform convergence there exists a sequence $(x_\varepsilon, t_\varepsilon)$ converging to $(x_0, t_0)$ such that $u^\varepsilon - \varphi$ has a local maximum at $(x_\varepsilon, t_\varepsilon)$.  Now if we we let 
\[ \tilde \varphi(x, t) := \varphi(x, t) - \varphi(x_\varepsilon, t_\varepsilon) + u^\varepsilon(x_\varepsilon, t_\varepsilon)\]
Then $u^\varepsilon - \tilde \varphi$ has a local maximum at $(x_\varepsilon, t_\varepsilon)$ with $(u^\varepsilon-\tilde \varphi)(x_\varepsilon, t_\varepsilon) = 0$.  We can now conclude by taking the limit of the viscosity subsolution property of $u^\varepsilon$. 

\medskip      
2.  Next we show that $u$ is a supersolution.  Let $\varphi$ be a classical free--boundary subsolution such that $\varphi(x, t_1) \prec u(x, t_1)$.  Since the $u_\varepsilon$'s are strictly ordered, $u < u_\varepsilon$ and hence $\varphi(x, t_1) \prec u(x, t_1) < u_\varepsilon(x, t_1)$.  Now suppose $\varphi$ touches $u_\varepsilon$ at some point $(x_2, t_2)$, then $\varphi(x_2, t_2) > 0$ since $u_\varepsilon$ is positive, so by continuity, there is a parabolic neighborhood of $(x_2, t_2)$ in which both functions are classical and positive.  By the Strong Maximum Principle, the touching cannot have occurred at $(x_2, t_2)$, a contradiction.  We conclude that $\varphi < u_\varepsilon$ so that in the limit $\varphi \leq u_\varepsilon$.

\end{proof}
An immediate consequence of above lemma is that weak solutions are viscosity solutions (see Corollary \ref{weak}). We shall introduce the precise notion of weak solutions in the next subsection, and summarize some results from \cite{BH}.

\subsection{Weak Solutions}
To be consistent with the setup in both \cite{BH} and \cite{jose_et_al}, let us return to the density variable and consider the solution of \eqref{main_eq_density} in a bounded domain $\Omega$ with Neumann boundary condition: 
$$
\left\{\begin{array}{lll}
\rho_t = \triangle \rho^m + \nabla \cdot (\rho \nabla \Phi) &\hbox{in }& \Omega \times \mathbf R^+,\\ \\
 \frac{\partial}{\partial \nu} \rho^m + \rho \frac{\partial \Phi}{\partial \nu} = 0 &\mbox{ on }& \partial \Omega \times \mathbf R^+,\\ \\
 \rho(x, 0) = \rho_0(x) &\hbox{ in }&\Omega.
\end{array}\right. \leqno (N)
$$
We will see shortly that we need not worry about the fact that we are on a bounded domain, but for now we will let $Q = \Omega \times \mathbf R^+$ and $Q_t= \Omega \times (0, t]$.  Following \cite{BH},  
\begin{defn}
We say $\rho: [0, \infty) \rightarrow L^1(\Omega)$ is a \emph{weak solution} of (N) if\\
 
(i) $\rho \in C([0, t]; L^1(\Omega)) \cap L^\infty(Q_t) \mbox{ for all } t \in (0, \infty)$;\\

(ii) for all test functions $\varphi \in C^{2,1}(\overline Q)$ such that $\varphi \geq 0$ in $Q$ and $\partial \varphi/\partial \nu = 0$ on $\partial \Omega \times \mathbf R^+$, we have 
$$
\int_\Omega \rho(t) \varphi(t) = \int_\Omega \rho(0) \varphi(0) + \int \int_{Q_t} (\rho \varphi_t + \rho^m \triangle \varphi - \rho \nabla \Phi \cdot \nabla \varphi)
$$
We also define a weak subsolution (respectively supersolution) by (i) and (ii) with equality replaced by $\leq$ (respectively $\geq$).
\end{defn}

>From \cite{BH} we have existence, regularity, uniqueness and comparison principle for weak solutions:
\begin{thm}[From \cite{BH}]\label{weak_soln_results}
Under the assumption that  $\Phi$ is $C^2$ in $\bar{\Omega}$ 
\begin{itemize}
\item[(a)] The problem (N) has a unique solution;\\
\item[(b)] The solution is uniformly bounded in $Q$ and is continuous in any set $\overline \Omega \times [0, T]$.\\
\item[(c)] Suppose $\underline \rho(t)$ is a subsolution and $\overline \rho(t)$ is a supersolution, then if $\underline \rho_0 \leq \overline \rho_0$ in $\Omega$, then $\underline \rho(t) \leq \overline \rho(t)$ in $\Omega$ for $t \geq 0$. 
\end{itemize}  
\end{thm}

The existence of solutions is obtained as the uniform limit of solutions to uniformly   problems (equicontinuity is obtained from \cite{DiB}).  For our purposes, a very simple approximation basically suffices and we summarize the relevant result in the following:

\begin{lemma}[From \cite{BH}]\label{equicontinuity}
Let $u^\varepsilon$ be a solution of (N) with initial data $u_0^\varepsilon = u_0 + \varepsilon$, then $u^\varepsilon$ is equicontinuous and there exists a subsequence which uniformly converge to $u$ which is the unique weak solution to (N) with initial data $u_0$.  
\end{lemma}

While \emph{a priori} our viscosity solution is defined in all of $\mathbb R^n$, since (formally at least) solutions of (PME--D) should have finite speed of propagation, the boundary conditions should be inconsequential if we take $\Omega$ sufficiently large. (Later we will also establish finite propagation for viscosity solutions -- see Corollary \ref{finite_propagation}.)  Control on the speed of expansion of the support can be done via comparison with any (weak) supersolution.  In particular, when $\Phi$ is {\it monotone} (that is, when $|D\Phi|>0$ except at one point where $\Phi$ achieves its minimum), we can use the stationary profiles of the form $\Psi(x) = (C-\Phi)_+$ with sufficiently large $C$ as a supersolution (see Theorem~\ref{equilibrium_solution}).  

\begin{remark}
Alternatively (and perhaps this is a cleaner line of reasoning), we can directly use the result of \cite{JGJ} on existence and uniqueness of solutions in all of $\mathbb R^n$, which implies in particular that the results of \cite{jose_et_al} also apply in that setting.
\end{remark}

Combining Lemma \ref{stability} with Lemma \ref{equicontinuity} and the uniqueness statement in Theorem \ref{weak_soln_results}, we obtain:
\begin{cor}\label{weak}
Any weak solution is also a viscosity solution.  
\end{cor}

We will eventually establish uniqueness of viscosity solutions via maximum--principle type arguments, which culminates in the identification of the two notions of solution.

\subsection{Construction of test functions }\label{barriers}

In this subsection we collect some test functions, i.e., (classical free boundary) sub-- (super) solutions, to (PME--D) which will be useful for comparison purposes.  In the first couple of lemmas (Lemmas \ref{parabolic_super_barrier} and \ref{parabolic_sub_barrier}), the idea is to control the $\Phi$ dependence via Taylor expansion in a small neighborhood of a point, so that we can appropriately perturb the test functions for (PME) constructed in \cite{BV} and \cite{CV} for our purposes.  

The starting point is to observe that if we consider (PME--D) in some small cylinder $Q_\alpha := B_\alpha(x_0) \times [t_0- \alpha, t_0+\alpha]$ and define $v_1(x, t) = \alpha^{-1}u(\alpha (x-x_0), \alpha (t-t_0))$, then $v_1$ satisfies, in the unit cylinder $B_1(0)\times [-1,1]$, an equation of the type   
\[(v_1)_t = (m-1) v_1\triangle v_1 + |D v_1|^2 + \vec{b} \cdot D v_1+ O(\alpha) D v_1 + \alpha(m-1)v_1\Delta\Phi,\]
where $\vec{b} = \nabla \Phi(x_0, t_0)$. The size of the last two terms  depends on the $C^2$-norm of $\Phi$ in $Q_\alpha$.  

Next we take 
\begin{equation}\label{exp_transform} v(x, t) = v_1(x - \vec{b}t, t)\end{equation}
Then $v$ satisfies
\begin{equation}\label{eq_v} v_t = (m-1) v\triangle v + |D v|^2 + O(\alpha)Dv +O(\alpha) v.
\end{equation}

Next proposition illustrates the necessary perturbation one needs to perform on solutions of (PME) to arrive at \eqref{eq_v}. 

\begin{prop}\label{prop:perturb}
Let $u(x,t)$ be a viscosity subsolution of  (PME) in $B_{1+\alpha}(0)\times [-1,1]$. Then for $0<\alpha<1$,
 $$
 u_1(x,t) := \sup_{y\in B_{\alpha-\alpha t} (x)}  e^{-\alpha t}u(y,t)
 $$ is a subsolution of
$$
(u_1)_t =(m-1) u_1\triangle u_1+ |D u_1|^2 - \alpha|Du_1| - \alpha u_1\leqno \mbox{(PME--sub)}
$$
in $B_1(0)\times [-1,1]$.
Similarly, if $u(x,t)$ is a viscosity supersolution of (PME) in $B_1(0)\times [-1,1]$, then for $0<\alpha<1$,

$$
u_2(x,t):= \inf_{y\in B_{\alpha-\alpha t} (x) } e^{\alpha t} u(y,t)
$$ is a supersolution of 
$$
(u_2)_t = (m-1)u_2\Delta u_2 + |Du_2|^2 +\alpha|Du_2|+\alpha u_2 \leqno \mbox{(PME--super)}
$$
in $B_1(0)\times [-1,1]$.
\end{prop}

\begin{proof}
We only show the supersolution part. Let $u_2$ be as given above.
Suppose then that $\varphi$ is classical and touches $u_2$ from below at some point $(x_0, t_0)$.  We first note that there exists $(x_1, t_1) \in \overline B_{\alpha - \alpha t} (x_0, t_0)$ such that $u_2(x_0, t_0) = e^{\alpha t_1}u(x_1, t_1)$.

Next for any unit vector $\hat{b}$, let us consider 
\[\tilde \varphi(x, t) = e^{-\alpha (t-(t_1 - t_0))}\varphi(x - (x_1 - x_0) - \alpha \hat{b} (t - t_0),  t - (t_1 - t_0)).\]
Then we note that 1) $\tilde \varphi(x_1, t_1) = \varphi(x_0, t_0)$ and so $(u- \tilde \varphi)(x_1, t_1) = 0$ and 2) by the definition of $u_2$ as an infimum and by continuity of $\varphi$, in a small parabolic neighborhood of $(x_0, t_0)$, it is the case that $u - \tilde \varphi \geq u_2 - \tilde \varphi \geq 0$; we therefore conclude that $\tilde \varphi$ touches $u$ from below at $(x_1, t_1)$ and so we have, 
$$
\begin{array}{lll}
[\varphi \triangle \varphi + |\nabla \varphi|^2](x_0, t_0) &=&e^{\alpha t_1} [\tilde \varphi \triangle \tilde \varphi + |\nabla \tilde \varphi|^2](x_1, t_1) \\ \\
&\leq& e^{\alpha t_1}\tilde \varphi_t(x_1, t_1) \\ \\
&=&[\varphi_t -\alpha\varphi-\alpha \hat b \cdot \nabla \varphi](x_0, t_0).
\end{array}
$$  
Now the desired inequality is achieved by setting $\hat b = \frac{\nabla \varphi}{|\nabla \varphi|}(x_0, t_0)$.

Indeed the above calculation shows that if $u$ is a viscosity supersolution of (PME), then $u_2$ should be a viscosity supersolution of (PME--super): If  a classical free boundary subsolution $\varphi$ of (PME--super) crosses $u_2$ from below, then the corresponding $\tilde \varphi$ is a subsolution of (PME) and crosses $u$, yielding a contradiction (there is no distinction between the interior and boundary cases).   
\end{proof}

\begin{lemma}\label{sphere_symm} [CV]
Consider the function 
\[ H(x, t; A, \omega) = A(|x| + \omega t - B)_+\]
with $R/2 < B < R$.  Then $u$ is a classical free boundary supersolution of (PME) in the domain $\{|x|\leq R\}\times [\omega^{-1}(R-B),0]$ if 
\[ \frac{\omega}{A} > 1 + 2(m-1)(d-1)\frac{R-B}{R}.\]
\end{lemma}

Proposition~\ref{prop:perturb} and Lemma~\ref{sphere_symm} yields the following:

\begin{cor}\label{parabolic_super_barrier}
Let us fix $x_0\in\R^n$ and let $H$ be given as in Lemma \ref{sphere_symm}.
Then the inf convolution of $H$, given as 
$$
\underline H(x, t;\alpha) = e^{\alpha t}\inf_{y\in B_{\alpha - \alpha t}(x)} H(y, t) 
$$
is a classical (free boundary) supersolution of \eqref{eq_v}. Consequently,  there exists $C=C_0$ which only depends on the $C^2$-norm of $\Phi$ in $B_1(x_0)$ such that 
$$
\tilde{H}(x, t) := \alpha \underline H(\alpha^{-1}(x-x_0+\vec{b}(t-t_0)),\alpha^{-1}(t-t_0);C\alpha)
$$ is a classical (free boundary) supersolution of (PME--D) in  $Q_\alpha:= B_\alpha(x_0) \times [t_0- \alpha, t_0]$.
\end{cor}
 \begin{proof}
 By Lemma \ref{classical_visc} and Proposition \ref{prop:perturb}, $\underline H$ is a viscosity supersolution of (PME--super), so it is sufficient to show that it has the required regularity.  For this, note that for simplicity we have only taken the supremum over space and the reader can readily check that in this case the infimum for $\underline H(x,t;\alpha)$ is achieved at the point $y$ which minimizes $|y|$ subject to the constraint that $|y - x| = \alpha - \alpha t$, and thus an explicit expression is possible for $\underline H$. 
 \end{proof}

\begin{remark}\label{remark}
In fact due to the explicit form of $H$ it follows that the free boundary velocity of $\tilde{H}$ is given by 
$$
V=\omega +\vec{b}\cdot \frac{D H}{|D H|}+C\alpha.
$$
\end{remark}

By comparison with these supersolutions, we immediately obtain 
\begin{cor}\label{finite_propagation}
Any viscosity solution has finite propagation speed and is bounded in a big ball in any local time interval.  Further, if $\Phi$ is convex, then via comparison with the stationary solutions of the form given in Theorem \ref{equilibrium_solution}, the above holds globally in time.
\end{cor}

For the analysis in section 3, we will make use of the Barenblatt profiles (see e.g., [CV] and [V]).

\begin{lemma}\label{barenblatt}[Barenblatt]
Let $B(x, t; \tau, C)$ be the family of functions 
\[ B(x, t; \tau, C) = \frac{(C(t+\tau)^{2\lambda} - K|x|^2)_+}{(t+\tau)}\]
with constants $\lambda, K, C, \tau > 0$ such that 
\[ \lambda = ((m-1)d + 2)^{-1}, ~~~2K = \lambda.\]
Then $B(x, t; \tau, C)$ is a classical (free boundary) solution of (PME).
\end{lemma}
Using Proposition~\ref{prop:perturb} (see also the proof of Corollary \ref{parabolic_super_barrier}) once again, we obtain the following:
\begin{lemma}\label{parabolic_sub_barrier}
Let us fix $x_0\in\R^n$ and let $B$ be a Barenblatt function.  
Then there exists $C$ which only depends on the $C^2$-norm of $\Phi$ in $B_1(x_0)$ such that 
 $$
 \psi(x, t) = \alpha e^{-C\alpha (t-t_0)}\sup_{y\in B_{C\alpha-C(t-t_0)}(x)}B(\alpha^{-1}(y-x_0+\alpha \vec{b}(t-t_0)), \alpha^{-1}(t-t_0))
 $$ is a classical (free boundary) subsolution of (PME--D) in $Q_\alpha:= B_\alpha(x_0) \times [t_0- \alpha, t_0]$.
\end{lemma}

\begin{remark}\label{perturbed_velocity}
The reason for taking the hyperbolic scaling is because we will have occasion to require rather fine control on the boundary velocity (see Lemma \ref{nontangential}) and this is the scaling which preserves the velocity -- in contrast to the parabolic scaling, which dramatically reduces the effect of the drift $\Phi$ in the bulk (the positivity set), but unfortunately at the cost of severely disrupting the boundary velocity.
\end{remark}

To establish the Comparison Principle, we will need the following weak analogue of (ii) in the definition of viscosity supersolutions for subsolutions, the proof of which utilizes an approximation lemma from \cite{BV}.

\begin{lemma}\label{no_cross_from_above}
Let $u$ be a viscosity subsolution of $(PME$--$D)$, and let $\varphi$ be a classical free boundary supersolution from Lemma \ref{parabolic_super_barrier} which lies above $u$ at some time $t_0$.  Then $\varphi$ cannot cross $u$ from above at a later time $t > t_0$.
\end{lemma}
\begin{proof}
Let $\varphi$ and $u$ be as described in the statement, and suppose that $\varphi$ touches $u$ from above at some point $(x_0, t_0)$.  From Lemma \ref{parabolic_super_barrier}, we have that $\varphi$ is given as the inf convolution of some spherical traveling waves from Lemma \ref{sphere_symm}, which we denote $\psi$.  Further, let us suppose the infimum is achieve at $(x_1, t_1)$ so that 1) $\varphi(x_0, t_0) = \psi(x_1, t_1)$ and 2) by the definition of $\varphi$ as an inf convolution, the translated function 
$$
\tilde \psi(x, t) = \psi(x + (x_1 - x_0), t + (t_1 - t_0))
$$ also touches $u$ from above at the point $(x_0, t_0)$.  From Lemma 4.4 in \cite{BV}, we know that $\psi$ can be given as the monotone limit of classical positive supersolutions, and hence the same is true of $\tilde \psi$: I.e., there exsits $\psi_\varepsilon \searrow \tilde \psi$, with $\psi_\varepsilon > 0$ classical.  But since $u$ cannot touch $\psi_\varepsilon$ by the Strong Maximum Principle, we obtain in the limit that $u \leq \tilde \psi$, which is a contradiction.
\end{proof}

\subsection{Comparison Principle and Identification with Weak Solution}

Here the outline of the proof closely follows that of the corresponding result for (PME) (Theorem 2.1 in \cite{CV}): We will give an abridged version of the proof, pointing out main steps and modifications for our problem.

\begin{thm}\label{thm:cp} [Comparison Principle]
If $u$ is a viscosity subsolution and $v$ is a viscosity supersolution in the sense of Definition \ref{viscosity} with strictly separated initial data, $u_0 \prec v_0$, then $u(x, t) \leq v(x, t)$ for every $(x, t) \in Q$. 
\end{thm}
\begin{proof}
1) [Sup and Inf Functions]  For given $\delta > 0$ and $r > 0$ small with $r \gg \delta$, we introduce the regularized functions 
\[ W(x, t) = \inf_{\overline B_{r - \delta t}(x, t)} v(y, \tau)\]
and 
\[ Z(x, t) = \sup_{\overline B_r(x, t)} u(y, \tau)\]
First  note that $W$ and $Z$ preserve properties of $v$ and $u$: 
\begin{itemize}
\item [$\circ$] \emph{$W$ is a supersolution and $Z$ is a subsolution;}
\item [$\circ$] \emph{$Z(\cdot, r) \prec W(\cdot, r)$ for $r$ sufficiently small.}
\end{itemize}  
For a proof of the first item see \cite{BV}, Lemma 7.1 or the proof of Lemma \ref{parabolic_super_barrier}).  The proof of the second item relies on the fact that 
\begin{itemize}
\item [$\circ$] \emph{The support of a viscosity subsolutions and supersolutions evolve in a continuous way.  Here continuity is understood as continuity in the Hausdorff distance (in time) of the positivity set.}
\end{itemize}
The proof of this can be done by comparison with the supersolutions (respectively subsolutions) constructed in Lemma \ref{parabolic_super_barrier} (respectively Lemma \ref{parabolic_sub_barrier}).  We omit the details since with replacement of barriers it is no different from the proof of Proposition 6.2 in \cite{BV}.    

Thus if we can prove that $W$ stays above $Z$ for all choices of $r$ and $\delta$(sufficiently small), then we may take $\delta \rightarrow 0$ and then $r \rightarrow 0$ to recover the conclusion for $u$ and $v$.  First let us note that due to the Strong Maximum Principle, $W$ cannot touch $Z$ from above, and therefore we are reduced to the analysis of a first contact point of $W$ and $Z$ at some $P_0 = (A, t_0)$. 

The key usefulness of $Z$ and $W$ lies in the fact that they enjoy interior/exterior ball properties: 
\begin{itemize}
\item [$\circ$] \emph{The positivity set of $Z$ has the interior ball property with radius $r$ at every point of its boundary \emph{and} at the points of the boundary of the support of $u$ where these balls are centered we have an exterior ball;}
\item [$\circ$] \emph{The positivity set of $W$ has the exterior ball property with radius less than $r - \delta t$ (since in this case we really have an exterior ellipsoid in space--time) and at the points of the boundary of the support of $v$ where these balls are centered we have an interior ball.}
\end{itemize}
For detailed proofs of these statements we again refer the reader to \cite{BV}. 

2) [The Contact Point] The first contact point $P_0 = (x_0, t_0)$ is located at the free boundary of both functions.  Therefore by the definitions of $Z$ and $W$, there is a point $P_1 = (x_1, t_1)$ on the free boundary of $u$ located at distance $r$ from $P_0$ and there is another point $P_2 = (x_2, t_2)$ on the free boundary of $v$ at distance $r_0 = r - \delta t_0$ from $P_0$.  Let us also denote by $H_Z$ (respectively $H_W$) the tangent hyperplane to the free boundary of $Z$ (respectively $W$) at $P_0$.  (see Figure \ref{fig_contact_point})
\begin{figure}
\center{\includegraphics[height=3.5in, width=3.5in]{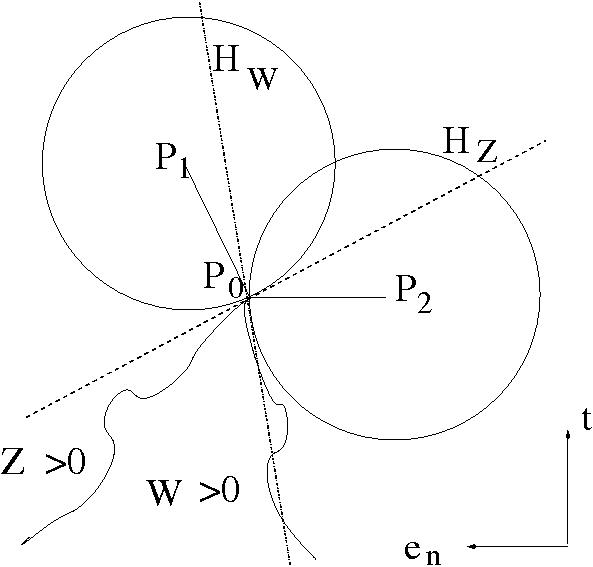}} \caption{The geometry at the contact point}
\label{fig_contact_point}
\end{figure}

\begin{lemma}\label{speed_well_defined}
Neither $H_Z$ nor $H_W$ is horizontal. In particular, one can denote the space-time normal vector to $H_Z$, in the direction of $P_1-P_0$, as $(e_n,m)\in\R^n\times \R$ where $|e_n|=1$ and $-\infty < m<\infty$.
\end{lemma}
\begin{proof}
It is enough to show that $t_1>t_0-r$ (i.e., $\Gamma(Z)$ cannot propagate with infinite speed) and $t_2 <t_0+r $ (i.e., $\Gamma(W)$ cannot propagate with negative infinite speed).  The desired conclusion then follows by the ordering of $Z$ and $W$.

We first show that $t_2<t_0+r$.  Otherwise $H_W$ is horizontal and after translation we have $P_0 = (0, -r)$ and $P_2 = (0, 0)$.  Moreover $\Omega(v)$ has an interior ball at $P_2$ with horizontal tangency with radius $0 < r^\prime < r$.  Now in any parabolic cylinder 
$$
C_{\eta} = \{ (x, t): |x| \leq \eta, -\eta^2 \leq t \leq 0\}
$$
 with bottom edge contained in the interior ball (which can be achieved by taking $\eta \leq r^\prime$), we have by continuity that $v \geq M > 0$ on that edge. (see Figure \ref{fig_horizontal_tangent})
\begin{figure}
\center{\includegraphics[height=3.5in]{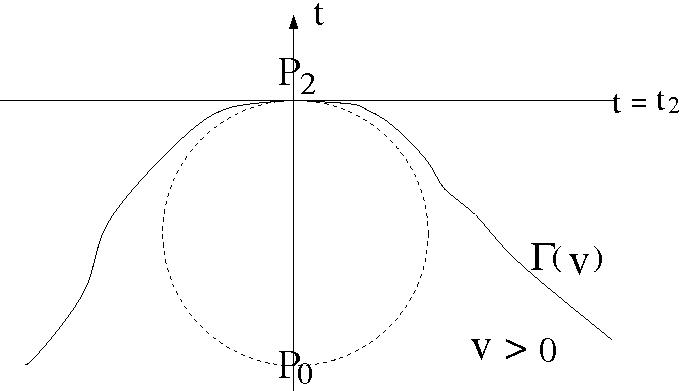}}
\caption{Infinite, negative speed}\label{fig_horizontal_tangent}
\end{figure}
 
On the lateral boundary of $C_\eta$ it may be the case that $v = 0$, so we will have to compare with a subsolution with support strictly contained in $-\eta < |x| < \eta$ at time $t=-\eta^2$ and still contains $0$ in its support at time $t = 0$, which rigorously implies that $v$ cannot contract sufficiently fast for $(0, 0)$ to be a free boundary point.  The necessary subsolution can be constructed as the one in Lemma \ref{parabolic_sub_barrier}, adjusted for the parabolic scaling.

The case $t_1>t_0-r$ follows \emph{mutatis mutantis} from the arguments in \cite{CV} and \cite{BV} (see \cite{CV}, Lemma 4.2 or \cite{BV}, Lemma 8.2) with the barrier in Lemma \ref{parabolic_super_barrier} replacing the barriers used therein.  
\end{proof}

3) [Non--tangential Estimate] 
The next lemma states that the normal velocity $V$ of $\Gamma(Z)$ at $(x_0,t_0)$ satisfies, in the viscosity sense,
$$
V_{(x_0,t_0)} \leq (|DZ| +D\Phi(x_0)\cdot\frac{DZ}{|DZ|})(x_0,t_0).
$$

\begin{lemma}\label{nontangential}
Let $x_n:= x\cdot e_n$, and consider a non-tangential cone $K:=\{x: x_n \geq k|x| \mbox{ with $k>0$}\}$.  Then we have 
$$
\lim_{x\in K,~x \to 0}\dfrac{Z(x_0+x, t_0)}{x_n} \geq m- D\Phi(x_0)\cdot e_n.
$$
\end{lemma}
\begin{proof}
The argument is parallel to the proof of Lemma 4.3. in \cite{CV}; the only difference for us is taking into account the change of reference frame introduced by the drift given by $\Phi$. 
This is ensured by the local nature of the construction of our barrier in Lemma \ref{parabolic_super_barrier}, which replaces the corresponding barriers used in \cite{CV}.
 
\end{proof}

4) [Conclusion] Due to Lemma \ref{nontangential}, we may place a small subsolution $\varphi$ from Lemma \ref{parabolic_sub_barrier} below $Z$ at $P_0$ with speed close to $m$ (again see Remark \ref{perturbed_velocity}, which assures us that our subsolutions are constructed so that this is possible) such that it crosses anything with speed $m^\prime < m$.  Since $\varphi$ is also below $W$ and hence $v$ (after a small translation), $v$ must expand by at least $m^\prime$, but then $\Gamma(W)$ has speed $m^\prime + \delta > m$ at $P_0$, yielding a contradiction to the fact that $Z$ touched $W$ from \emph{below} at $P_0$.
\end{proof}

We can now establish uniqueness of viscosity solutions:

\begin{thm}\label{exists_unique}
The problem (PME--D) admits a unique solution in the class of viscosity solutions as defined in Definition \ref{viscosity} for continuous and nonnegative initial data.  This solution coincides with the continuous weak solution. 
\end{thm}

\begin{proof}  
The existence of a continuous weak solution can be provided as the uniform limit of classical solutions with initial data $u_{0, \varepsilon} = u_0 + \varepsilon$, and by Lemma \ref{stability}, such a limit, which we will denote by $U$, is also a continuous viscosity solution.  Further, by comparison with $u_{0, \varepsilon}$ and taking a limit, it is clear that  such a limit $U$ is also a maximal viscosity solution.  

Uniqueness would follow if we can show that any other viscosity solution $u$ also cannot be smaller than $U$.
For this purpose, consider $u_n(x,t)$ with initial data $u_n(x,0):= (u_0-\frac{1}{n})_+$.    Now consider positive $u_n^{\e_n}$ such that $|u_n^{\e_n}-u_n| <\frac{1}{n}$ in $\R^n\times [0,T]$. It follows from Lemma \ref{equicontinuity} that $u_n^{\e_n}$  uniformly converges to $U_2(x,t)$, which is then a continuous weak solution of (PME--D).  Therefore, by uniqueness of weak solutions, $U_2$ is equal to $U$.  On the other hand by Theorem~\ref{thm:cp} $u_n \prec u$ and thus $U=U_2 \leq u$. Hence we conclude.
\end{proof}

Using Theorem \ref{thm:cp} and Theorem \ref{exists_unique}, we can in fact prove a stronger comparison theorem for viscosity solutions (see \cite{BV}, Theorem 10.2):

\begin{thm}\label{cp_strong}
Let $u_1$ and $u_2$ be respectively a viscosity subsolution and a viscosity supersolution of (PME--D) in some parabolic cylinder $Q$ with initial data $u_{0, 1}$ and $u_{0, 2}$ such that $u_{0, 1}(x) \leq u_{0, 2}$.  Then $u_1(x, t) \leq u_2(x, t)$ for all $(x, t) \in Q$.
\end{thm}

This comparison theorem in particular allows us to restrict attention to only the (classical free boundary) supersolutions used to establish Theorem \ref{thm:cp}, and consequently, we can now strength Lemma \ref{no_cross_from_above} to enable comparison with any classical free boundary supersolution:

\begin{lemma}
Let $u$ be a viscosity subsolution of (PME--D), and let $\varphi$ be a classical free boundary supersolution which lies above $u$ at some time $t_0$.  Then $\varphi$ cannot cross $u$ from above at a later time $t > t_0$.
\end{lemma}
\begin{proof}
Let us first replace $\varphi$ by a viscosity solution of (PME--D) with the same initial data which we denote $v$.  Since $\varphi$ is a viscosity supersolution by Lemma \ref{classical_visc}, we have that $\varphi \geq v$ by Theorem \ref{cp_strong}. Finally, $u \leq v$ by Theorem \ref{thm:cp}.
\end{proof}

\section{Convergence to Equilibrium}
We begin by discussing the set of equilibrium solutions to (PME--D) and reviewing some known results.  Since by Theorem \ref{exists_unique}, the unique viscosity solution coincides with the continuous weak solution, we may carry out our discussion in the context of weak solutions.    

The set of equilibrium solutions and uniform convergence of solutions to the equilibrium are established in \cite{BH}. Below we state the corresponding result in the pressure variable.
\begin{thm}\label{equilibrium_solution} [Theorem 5.1, \cite{BH}]
The set of equilibrium solutions for (PME--D) is given by 
\[ \begin{split}S = \{W \in C(\mathbb R^n): \mbox{$W \geq 0$ in $\Omega$}, \mbox{and for every $x \in \Omega$, either $W(x) = 0$}\\ 
\mbox{or  $W + \Phi = C$ for some constant $C$ in a neighborhood of $x$}\}\end{split}\]
Further, given $u_0$, there exists a unique $W(x) \in S$ such that $u(x,t)$ uniformly converges to $W$ as $t\to\infty$.
\end{thm}

It is fairly immediate that $S$ is contained in the set of equilibrium solutions; the converse containment and the convergence statement are established based on a $L^1$ contraction result in \cite{BH}.  

Under the assumption that $\Phi$ is convex. Note that the density function  $\rho(\cdot,t)$ given in (\ref{density}) preserves its $L^1$ norm over time. Therefore there is a unique equilibrium solution $u_{\infty}=(C_0-\Phi)_+$ to which $u(\cdot,t)$ uniformly converges as $t\to\infty$, i.e., the one with 
$$
\int (u_0)^{1/(m-1)}(x) dx = \int (u_{\infty})^{1/m-1}(x) dx.
$$

An explicit exponential rate of convergence is then derived in \cite{jose_et_al} by the entropy method:
\begin{thm}\label{Jose}
Suppose that $\Phi$ is strictly convex, i.e., there exists a constant $k_0>0$ such that   $x^T \cdot [(\hess \Phi(x)) x] \geq k_0 |x|^2$ for $x\in\R^n$.  Let $u_{\infty}=(C_0-\Phi)_+$ be the equilibrium solution to which our solution $u(x,t)$ converges as $t\to\infty$. Then there exist constants $K, \alpha > 0$ depending on $m$, $k_0$ and the $L^1$ norm of $u_0$ such that  
$$
\int |u(x,t)-u_\infty(x)| dx  \leq Ke^{-\alpha t}.
$$
\end{thm}  

\begin{remark} In fact the estimate in \cite{jose_et_al} is given in terms of the pressure variable $\rho$. Due to Corollary ~\ref{finite_propagation},  for a convex (and in fact monotone) potential $\Phi$ $u$ is uniformly bounded with its support contained in a compact set for all times. This allows us to derive the estimate for $u$ from that of $\rho$ for $1<m<\infty$. Further, due to the equivalence of all $L^p$ norms in our setting, we will take some liberties in passing between $u$ and $\rho$ in our estimates.
 \end{remark}
 
 \subsection{Convex potential}

As an application of the viscosity solutions theory, we will convert the $L^1$ estimate in Theorem \ref{Jose} into a pointwise estimate (see Lemmas \ref{lowerbound} and \ref{diffusivity}); such an estimate in turn will yield a quantitative estimate on the rate of the free boundary convergence (see Theorem \ref{FB_conv_rate}).  

Rescaling $u$ by $cu(x, ct)$ if necessary, let us assume for the rest of this subsection that $u\leq 1$ and $\max \Delta \Phi \leq 1$ on our domain of consideration, which is bounded (see Corollary \ref{finite_propagation}) and we assume to be $\{|x|\leq R\}$ for some $R > 0$.

 \begin{lemma}\label{lowerbound}[Uniformly Bounded From Below]
There exists a sufficiently small constant $k>0$, depending only on $m$ and $n$, such that the following is true:

Suppose, for $(x_0,t_0)\in\R^n\times (0,\infty)$ and for $0<a<1$,
$$
a^{-n}\int_{B_a(x_0)} u(\cdot,t_0) dx \geq a^{k}.
$$

Then  $u(\cdot,t_0+a)\geq a^{k^\prime}$ in $B_a(x_0)$.

\end{lemma}

\vspace{10pt}

\begin{proof}
1. Let us define 
$$
\tilde{u}(x,t):= u(a(x-x_0),a^2(t-t_0)).
$$
Then $\tilde{u}$ is a supersolution of 
\begin{equation}\label{perturb}
\psi_t = (m-1)\psi\Delta \psi + |D\psi|^2 -Ca(|D\psi| + \psi),
\end{equation}
where $C$ is a constant depending on the $C^2$--norm of $\Phi$ (near $x_0$).  Below we will construct a subsolution of (\ref{perturb}) to compare with $\tilde{u}$ in order to establish the lemma.

\vspace{10pt}

2.  Let us consider $w(x,t)$ which satisfies
\begin{equation}\label{eqn1}
w_t = (\tilde{m}-1)w\Delta w +|Dw|^2 -Ca
\end{equation}
in the weak sense (see e.g., \cite{DibF}), with initial condition $w(x,0):=(1-a)\tilde{u}(x,0)\chi_{|x|\leq 1}$, where $\tilde{m}-1 = (1+a)(m-1)$. 

Then, say for $t\geq 1/2$, $w(\cdot,t)$ is H\"{o}lder continuous due to \cite{DibF}.
Since $u$ is bounded by $1$, so is $w(x,0)$, by the Maximum Principle.  Further, note that any solution of the (PME) is now automatically a supersolution of \eqref{eqn1} and therefore, using an appropriate Barenblatt profile as a supersolution of (\ref{eqn1}), one can check that 
\begin{equation}\label{bound}
\Omega_{1/2}(w) \subset \{|x| \leq 2\}.
\end{equation} 

 Moreover, integration by parts yields that 
$$
\int w(x,t) dx  = \int w(x,0) dx - at,
$$
and in particular we deduce that e.g., $\int w(x,1/2) dx \geq a^k/2$.  Let $x^*$ be the point where $w(\cdot,1/2)$ assumes its maximum, then from (\ref{bound}) we see that $w(x^*,1/2) \geq C_{n}a^k$ for some dimensional constant $C_n$.  Due to the H\"{o}lder regularity of $w$, 
$$
w(\cdot,1/2) \geq (C_n/2) a^k \hbox{ in } B_{a^{k_2}}(x^*).
$$
where $k_2 = \gamma^{-1} k$ where $0<\gamma<1$ depends only on $m$ and $n$.

Let $U(x,t):= B(x-x^*,t;\tau,C)$, where $B(x, t; \tau, C) = \frac{(C(t+\tau)^{2\lambda} - K|x|^2)_+}{(t+\tau)}$ is the Barenblatt profile given in Lemma \ref{barenblatt}, with $\tilde{m}$ as the permeability constant and the conditions 
$$
C\tau^{2\lambda-1} = a^k/4 \hbox{ (height), and } \sqrt{C/K} \tau^{\lambda} \leq a^{k_2}/2 \hbox{ (the size of initial support).}
 $$ 
 Note that, in its positive set, $U$ is concave.
  Therefore $\tilde{U}(x,t):=(U(x,t)-Cat)_+$ satisfies, in its positive set, 
 $$
 \begin{array}{lll}
 \tilde{U}_t  &=& U_t -Ca \\ \\
                       &=& (\tilde{m}-1)(U-Cat)\Delta U  + (\tilde{m}-1)at\Delta U +|DU|^2  -Ca \\ \\
                       &\leq& (\tilde{m}-1)\tilde{U}\Delta\tilde{U} + |D\tilde{U}|^2-Ca.
 \end{array}
 $$
 Therefore $\tilde{U}(x,t)$ is a subsolution of \eqref{eqn1} for positive $t$. Comparison with $\tilde{U}(x,t)$ and $w(x,t+\frac{1}{2})$ yields that 
$$
w(\cdot,a^{-1} ) \geq a^{k^\prime} \hbox{ in } B_1(0),
$$
if $k'$ is chosen sufficiently large.

\vspace{10pt}

3. Observe that  $v:= (1+a) w$ satisfies
$$
\begin{array}{lll}
v_t =  &\leq &(m-1)v\Delta v + \frac{1}{(1+a)} |Dv|^2 - Ca(1+a) \\ \\
 &\leq & (m-1)v\Delta v +|Dv|^2 -Ca(|Dv|^2+1)\\ \\
 &\leq & (m-1)v\Delta v +|Dv|^2 -Ca(|Dv|+1)
\end{array}
$$
Hence $v$ is a subsolution of \eqref{perturb}, and further  $v(x,0) = (1-a^2)\tilde{u}(x,0) \leq \tilde{u}(x,0)$. Therefore by the comparison principle (for weak solutions, see \cite{DibF}) we have $v\leq \tilde{u}$. From previous discussions we obtain that $\tilde{u}(\cdot,a^{-1}) \geq a^{k^\prime} \hbox{ in } B_1(0)$, and thus
$$
u(\cdot,a) \geq a^{k^\prime} \hbox{ in } B_a(0),
$$
as desired.

 \end{proof}

 Next lemma establishes a uniform upper bound on $u$ in terms of its $L^1$ norm.
 Here we would address our equation in pressure form, i.e., \eqref{main_eq_density}, to invoke regularity theory for divergence form operators studied in \cite{LSU}.

\begin{lemma}\label{diffusivity}
Let $K$ be a compact subset of $\R^n$ with $u=0$ outside of $K$ for all time.
Then there exists a constant $C>0$ depending on $m>1$, $\sup \rho$ and $\max \Delta \Phi:= M_1$ such that the following holds:

 If 
 $$
 \int_{B_C(0)} \rho(\cdot,t) dx \leq c_0\hbox{ for } t_1\leq t\leq t_2:=t_1+\log (1/c_0).
 $$ 
 then $\rho(\cdot,t_2)\leq Cc_0^{1/n+1}$ in $B_1(0).$
\end{lemma}

\begin{proof}

1. We proceed by induction.  Suppose that $u\leq a = c_0^{1/n+1} 2^k$ in $B_C(0)\times [0,t_2]$ where $k>0$ is chosen such that 
\begin{equation}\label{condition}
 4Cc_0^{1/n+1}\leq a\leq 1,
\end{equation}
with $C$ to be determined later. Our goal is to show that 
$$
u\leq a/2\hbox{ in }B_{C(1-a/2)}(0) \times [1,t_2].
$$
 Then the desired result is obtained by iteration, beginning with $a=1$ and continuing until $a$ reaches the lower bound in \eqref{condition}. Note that the total number of iteration for this process, therefore the total time we need for the desired result, is of order $\log(1/c_0)$.

2. Let $\rho_1(x,t)$ solve our equation  \eqref{main_eq_density} with initial data $\rho_0+a/10$ in $\Sigma:=B_C(0) \times [0,\infty)$ with boundary data corresponding to $u+a/10$.
Observe that $f(t) = \frac{11}{10}ae^{t}$ is a supersolution of  \eqref{main_eq_density}, since $\max \Delta \Phi = 1$ due to our normalization.
Thus by the comparison principle $\rho_1$ ranges from $a/10$ to $4a$ for $0\leq t\leq 1$. 
Therefore, treating the diffusion coefficients as \emph{a priori} given, $\tilde{\rho}(x,t):=a^{-1} \rho_1(ax,t)$ solves a quasi-linear equation of divergence form with diffusion coefficient of unit size:
$$
\tilde{\rho}_t = \nabla\cdot(b(x,t) \nabla\tilde{\rho}+\tilde{\rho} \nabla\Phi) ,\hbox{ where } b(x,t) = \tilde{\rho} \in [1/10,4]. 
\leqno(P)
$$

  In particular, we can decompose $\tilde{\rho}:= \tilde{\rho}_1 +\tilde{\rho}_2$ where $\tilde{\rho}_1$ solves (P) (which is linear) with initial data $\rho_0/a+1/10$ and boundary data zero, and $\tilde{\rho}_2$ solves (P) with initial data zero and boundary data $4$. We claim that both of them stays smaller than $1/4$ in $B_{C/a-1/4}(0)\times [0,1]$.

For $\tilde{\rho}_1$ we have
$$
\int_{B_{1/a}(0)} \tilde{\rho}_1 (\cdot,t) dx \leq \int_{B_{1/a}(0)} \tilde{\rho}_1(\cdot,0) dx = c_0a^{-(n+1)}\hbox{ for all } t>0, 
$$
 Due to \eqref{condition} and the H\"{o}lder regularity of $\tilde{\rho}_1$ (see \cite{LSU}) we have 
$$
\tilde{\rho}_1 (\cdot,1) \leq C c_0 a^{-(n+1)} \leq 1/4 \hbox{ in } B_{1/a-1}(0).
$$
if we choose $C>0$ sufficiently large in (\ref{condition}) corresponding to the H\"{o}lder regularity for solutions of (P) for $t\geq 1$.

As for $\tilde{\rho}_2$,  arguments with test functions in the weak formulation of (P) (not very different from the case of the heat equation) in combination with the H\"{o}lder regularity estimates yield that, for sufficiently large $C>0$,
$$   
\tilde{\rho}_2(x,t) \leq 4te^{-3/t} \hbox{ in } B_{C/a-C/4}(0)
$$
with a dimensional constant $C$. In particular $\tilde{\rho}_2(x,1) \leq 1/4 \hbox{ in } B_{C/a-C/4}(0).$

Hence we obtain $\tilde{\rho}(x,1)\leq 1/2$ in $B_{C/a-C/4}(0)$. Scaling back to our original density function $\rho(x,t)$, we conclude that
$$
\rho(\cdot,1)\leq a/2\hbox{ in }B_{C(1-a/4)}(0).
$$
Due to our assumption on $\rho_0$,  one can go through the above argument starting at any time $t=\tau\in [0,t_2]$ instead of $t=0$, obtaining that $\rho(\cdot,t) \leq a/2$ in $B_{C(1-a/2)}(0)\times [1,t_2]$.

5. Repeating the above argument with $a_2:= a/2$ starting at $t=2$, we get 
$$
\rho \leq a_2/2 =a/4\hbox{ in }B_{C(1-a/4-a/8)}(0)\times [3,t_2].
$$
 If we iterate up to of order $\log c_0$ times, then $a$ reaches the lower bound in (\ref{condition}), and we arrive at the desired result.
  
 \end{proof}
 
Next we use the uniform bounds obtained above to investigate the rate of free boundary convergence.

   \begin{thm}\label{FB_conv_rate}
  Let $\Phi$ and $u_{\infty}$ be as in Theorem~\ref{Jose}. Then for $T>0$ sufficiently large, $\Gamma(u)$ is in the $Ke^{-\alpha_2 t}$-neighborhood of $\Gamma(u_{\infty})$. Here $1/K, \alpha_2>0$ is a sufficiently small constant  depending on $m$, $\sup u_0$,  $k_0$, $M_1$, $A:=\min_{\Phi(x)>C_0} |D\Phi|$ and $n$. 
   \end{thm}
 
 \begin{proof}
0. We first show that $\Gamma_t(u)$ is close to the equilibrium profile from the inside, i.e.,   
\begin{equation}\label{inside}
\{x: d(x,\R^n-\Omega(u_\infty)) \geq Ce^{-Ct}\} \subset \Omega_t(u).
 \end{equation}

After $t=T \geq C\ln a$,
$L^1$-average of $u$ in $B_a(x_0)$ is bigger than $Ca^{k}$ when the center $x_0$ is in $\Omega_a:=\{\Phi\leq C_0-a^k\}$. Hence Lemma~\ref{lowerbound}  yields that 
$u(x_0,T) \geq C_2a^{k^{\prime}}$  with $T = a^{k/2}$. As a result we conclude that $\Gamma(u)$ lies outside of $\Omega_a$ after $t = C\ln a$.

It remains to show that $\Gamma_t(u)$ is also close to the equilibrium profile from the outside.  This will be more involved.

1. Let $D_a:=\{\Phi\geq C_0+a\}$ and let $b(t) := e^{-\alpha t}$.  Due to Theorem~\ref{Jose}, we have  
 $$
 \int_{D_{b(T)}}  \rho(x,t) dx\leq Ke^{-\alpha T}  \hbox{ for } t\geq T.
$$

 Take any point $x_0$ such that $B_{Cb(T)}(x_0)\subset D_{b(T)}$ if $T$ is chosen large. Hence  one can apply Lemma~\ref{diffusivity} to
    $$
    \tilde{u}(x,t):= u(b(T)(x-x_0), b(T)^2t+T)
    $$ with $t_1=0$ and $c_0 = Ke^{-\alpha T}$ to obtain
 $$
 \tilde{u}(\cdot,\alpha T+\beta) \leq K'e^{-\alpha T/(n+1)}\hbox{ in } B_1(0).
 $$
Since $x_0$ was arbitrarily chosen, repeatedly using Lemma~\ref{diffusivity} and scaling back we get
  \begin{equation}\label{upperbound}
  u(x,t) \leq K e^{-\alpha T/(n+1)} \hbox{ in } D_{2b(T)}\times [T_1,\infty), T_0:=(1+ C\alpha)T. 
  \end{equation}
  
  2.  We claim that estimate (\ref{upperbound}) yields that 
 \begin{equation}\label{outside}
 \Gamma_{t}(u) \subset \R^n-D_{a(t)}\hbox{ with } a(t):=Ke^{-\alpha_2 t}\hbox{ for } t\geq T_0
 \end{equation}
 where $\alpha _2=1/K^3$ and $K>\dfrac{2C_0}{A}$, where $C_0$ is given in Corollary~\ref{parabolic_super_barrier}. 
 We prove the claim  by induction. Due to Corollary~\ref{finite_propagation} the claim is true  up to a sufficiently large time $t=T_0$ if $K$ is sufficiently large (in particular $\alpha_2<\alpha$). Next let 
$$
T^*:=\sup\{t_0: (\ref{outside}) \hbox{ holds for } 0\leq t\leq t_0\} \geq T_0.
$$
We want to show that $T^*=\infty$.
Let us choose $x_0\in \partial D_a$, where $a =K e^{-\alpha_2T^*}$. Since $\Phi$ is convex with $|D\Phi|>0$ in $D_a$, there is an exterior ball $B_1(x_1)$ outside of $D_a$ such that $x_0\in \partial B_1(x_1)$. Due to (\ref{upperbound}) and due to the fact $T^*\geq T_0$, we have for $\tilde{a}:= 1/K$, 
$$
u \leq \tilde{a}/K\hbox{ in } \Sigma:=[B_{1+\tilde{a}}(x_1)\cap B_{4\tilde{a}}(x_0)]\times [T^*,\infty).
$$
Let us first make a heuristic argument. Suppose that $x_0\in\Gamma_{T^*}(u)$ (by the definition of $T^*$ such an $x_0$ exists). Then at $x_0$ we have $D\Phi(x_0,T^*)$ pointing in the direction of $x_1-x_0$, which is the inward normal of $\Gamma(u)$ at $(x_0,T^*)$, which is parallel to $-\frac{Du}{|Du|}(x_0,T^*)$. Hence formally the normal velocity of $\Gamma(u)$ at $(x_0,T^*)$ should be
\begin{equation}\label{velocity}
V =(|Du | - |D\Phi|)(x_0,T^*)  \leq O(\tilde{a})-A < -A/2 
\end{equation}
if $K>2/A$.  This way one can conclude that the furthest point of $\Gamma_t(u)$ from $D_0$ shrinks away with the normal velocity less than $-A/2$ at $t=T^*$. 
This would yield a contradiction to the definition of $T^*$ since
$$
-(Ke^{-\alpha_2 t})' = K\alpha_2 a  \leq A/2.
$$
(see Figure \ref{fig_final_thm})
\begin{figure}
\center{\includegraphics[height=4in,width=4in]{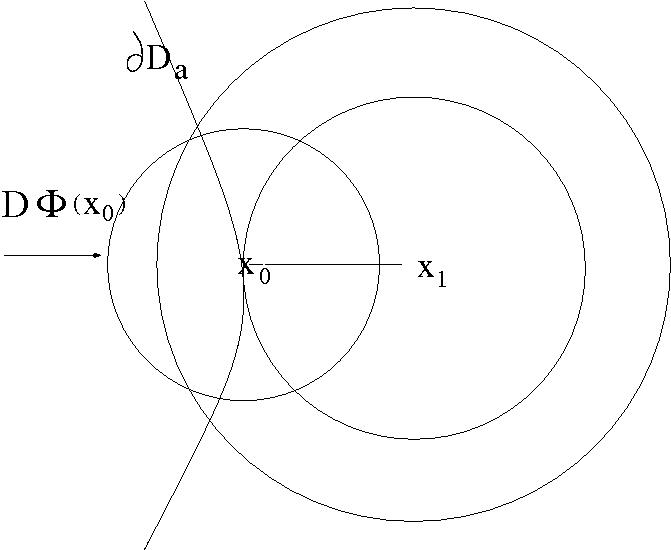}}\caption{The construction of barriers}\label{fig_final_thm}
\end{figure}

The rest of the proof consists of a barrier argument to establish an appropriate version of \eqref{velocity}.

3. Consider $H$ from Lemma \ref{sphere_symm} with $\omega = 8/K$, $A=4/K$, $B=3/4$, $R=1$ and define $\tilde{H}$ as in Lemma \ref{parabolic_super_barrier} with $\alpha=\tilde{a}$ and $(x_0,t_0)=(x_0,T^*)$. Then $\tilde{H}$ is bigger than $u$ on the parabolic boundary of $\Sigma\cap\{ T^*\leq t\leq T^*+\e\}$ for sufficiently small $\e>0$.
 We then compare $\tilde{H}$ with $u$ in $\Sigma\cap[T^*,T^*+\e]$. Observe that the spatial outward normal of $\Gamma(\tilde{H})$ at $(x_0,T^*)$, which is $-D\tilde{H}/{|D\tilde{H}|}$, points toward $x_1-x_0$. Hence the outward normal velocity of $\Gamma(\tilde{H})$  satisfies (see Remark \ref{remark}) 
$$
 V = \omega -|D\Phi|(x_0,T^*)+ C\tilde{a} \leq -A/2 .
$$
if $K>2C/A$. Since $u\leq \tilde{H}$ in $\Sigma\cap [T^*,T^*+\e)$, we conclude that
$$
u(\cdot,T^*+\e)=0\hbox{ in } B_{A\e/2}(x_0).
$$
Since $x_0\in\partial D_a$ is chosen arbitrarily, we conclude that $u(\cdot,T^*+\e)$ vanishes in $A\e/2$-neighborhood of $D_a$, which includes  $D_{a(T+\e)}$ if $K>2/A$ for any $\e>0$.

This contradicts the definition of $T^*$.

\vspace{10pt}
 
 \end{proof}

\subsection{Monotone potential}
Suppose $\Phi$ is  {\it monotone}, i.e., 
$$
|D\Phi| >0 \hbox{ except at } x=x_0 \hbox{ where } \Phi \hbox{ obtains its minimum.} 
 $$
 We are not able to yield quantitative estimates on the rate of free boundary convergence, due to the lack of available $L^1$--estimates. However we state the theorem below to illustrate that if we neglect the rate of convergence, then considerably simpler arguments already yield the free boundary convergence of $u$ as $t\to\infty$,

\begin{thm}\label{monotone}
Suppose $\Phi(x)$ satisfies  $|D\Phi|(x) >0$ except at $x=0$, where $\Phi$ achieves its minimum. Let $u_{\infty}(x):= (C_0-\Phi)_+(x)$ where $C_0$ is chosen such that $\|u_{\infty}\|_{L^{1/m-1}}=\|u_0\|_{L^{1/m-1}}$.
Then $\Gamma_t(u)$ uniformly converges to $\Gamma(u_\infty)$ in the Hausdorff distance, as $t\to\infty$.
\end{thm}

\begin{proof}
 \medskip
 
0. We first verify that for any compact subset $K$ of $\Omega(u_\infty)$ there exists $T>0$ such that $\Gamma_t(u)$ lies outside of $K$ for $t\geq T$.  This is immediate from the uniform convergence of $u$ to $u_{\infty}$.

1. It remains to establish the statement from  the outside of $\Omega(u_\infty)$. 

Since $\Phi$ is monotone the sets $\{x: \Phi(x) \leq r\}$ starts from a point  (when $r=\min \Phi$) and extends all the way to $\R^n$. In particular, there is only one equilibrium solution $u_\infty$ to which $u$ converges as $t\to\infty$. 
Choose $T>0$. Recall that, by the comparison principle, the set $\bar{\Omega}(u)\cap [0,T]$ is compact in $\R^n\times [0,\infty)$. 
Therefore there exists a largest  $C(T)>0$ such that  $\{x: \Phi(x) < C\}$ contains $\Omega_T(u)$. 
 We will show that $\Omega(u)$, after a while, eventually shrinks away so that after $T_1>T$ we have 
\begin{equation}\label{shrinkage}
\Omega_{T_1}(u) \subset\{x: \Phi(x,t) <C_2\}\hbox{ with } C_2\leq C.
\end{equation}
Let us choose $x_0\in\partial\{\Phi(x)<C\}$. Due to the regularity of $\Phi$ there is a ball $B_r(x_1)\subset\{\Phi(x) \geq C\}$ such that $x_0\in\partial B_r(x_1)$, where $r>0$ is independent of the choice of $x_0$.
Due to \cite{BH}, $u(x,t)$ uniformly converges to zero in $\{x: \Phi(x) \geq C\}$ with $C>C_0$.  Therefore there exists $T_0$ such that
 $$
 u \leq r/2 \min_{\Sigma} |D\Phi|(x)\hbox{ in }  \Sigma\times [T_0,\infty), \hbox{ where }\Sigma:=\{x: C_0<C_2 <\Phi(x) <C+1 \}.
 $$
(It is noted that the condition is to ensure that, formally, $|Du| < |D\Phi|$ at $(x_0,T_0)$).  We now argue as in step 3 of the proof of Theorem~\ref{FB_conv_rate} to construct  a radially symmetric barrier supported in $(B_{2r}(x_1)-B_{ r(t)}(x_1))\times [T_0, T_0+1]$ with $r'(t)<0$,  boundary data zero on $\partial B_{r(t)}(x_1)$ and $2r$ on $\partial B_{2r}(x_1)$, to demonstrate that at $\Gamma(u)$ stays out of $B_{r/2}(x_0)$ after $t=T_0+1$.  Since $x_1$ is an arbitrary point of $\partial\{\Phi(x) <C\}$, we arrive at (\ref{shrinkage}).

 \end{proof}

\section*{\large{Acknowledgments}} 
We thank Jose Carrillo for suggesting this problem and help discussions and communications.  We also thank IPAM for their hospitality during the Kinetic Transport program.   
The first author is supported by the  NSF DMS--0700732 and a Sloan Foundation fellowship.
The second author is supported by the UCLA Dissertation Year Fellowship and under NSF DMS--0805486.
\hspace{16 pt}

\end{document}